\documentclass[11pt]{amsart}
\usepackage{amssymb}
\usepackage{amsmath}
\usepackage{amsfonts}
\usepackage{graphicx}

\usepackage[total={17cm,22cm},top=2.5cm, left=2.3cm]{geometry}
\usepackage{hyperref}
    \usepackage{aeguill}
    \usepackage{type1cm}

\theoremstyle{plain}
\newtheorem{thm}{Theorem}

\newtheorem{corollary}[thm]{Corollary}

\newtheorem{definition}[thm]{Definition}

\newtheorem{lemma}[thm]{Lemma}

\newtheorem{proposition}[thm]{Proposition}

\newtheorem{theorem}[thm]{Theorem}

\newtheorem{rem}[thm]{Remark}
\newtheorem{ex}[thm]{Example}

\newtheorem{defn}[thm]{Definition}
\newcommand{\N}{\mathbb{N}}

\newcommand{\R}{\mathbb{R}}

\newcommand{\Rn}{\mathbb{R}^n}

\DeclareMathOperator{\lip}{Lip\,\!}
\DeclareMathOperator{\SBilip}{SBilip\,\!}

\DeclareMathOperator{\conv}{conv\,\!}

\usepackage[usenames,dvipsnames]{color}

\usepackage[dvipsnames]{xcolor}

\begin{document}

\title[Kirszbraun's theorem via an explicit formula]{Kirszbraun's theorem via an explicit formula}

\author{Daniel Azagra}
\address{ICMAT (CSIC-UAM-UC3-UCM), Departamento de An{\'a}lisis Matem{\'a}tico y Matem\'atica Aplicada,
Facultad Ciencias Matem{\'a}ticas, Universidad Complutense, 28040, Madrid, Spain.  {\bf DISCLAIMER:}
The first-named author is affiliated to Universidad Complutense de Madrid, but this does not mean this institution has offered him all the support he expected; on the contrary, the Biblioteca Complutense has hampered his research by restricting his access to many books.
}
\email{azagra@mat.ucm.es}

\author{Erwan Le Gruyer}
\address{INSA de Rennes \& IRMAR, 20, Avenue des Buttes de Co\"esmes, CS 70839
F-35708, Rennes Cedex 7, France}
\email{Erwan.Le-Gruyer@insa-rennes.fr}

\author{Carlos Mudarra}
\address{ICMAT (CSIC-UAM-UC3-UCM), Calle Nicol\'as Cabrera 13-15.
28049 Madrid, Spain}
\email{carlos.mudarra@icmat.es}

\date{Oct 01, 2018}

\keywords{Lipschitz function, Kirszbraun Theorem}

\subjclass[2010]{47H09, 52A41, 54C20}

\thanks{D. Azagra and C. Mudarra were partially supported by Grant MTM2015-65825-P and by the Severo Ochoa Program for Centres of Excellence in R\&D (Grant SEV-2015-0554).}

\begin{abstract}
Let $X,Y$ be two Hilbert spaces, $E$ a subset of $X$ and $G: E \to Y$ a Lipschitz mapping. A famous theorem of Kirszbraun's states that there exists $\widetilde{G} : X \to Y$ with $\widetilde{G}=G$ on $E$ and $\lip(\widetilde{G})=\lip(G).$ 
In this note we show that in fact the function
$$\widetilde{G}:=\nabla_Y(\textrm{conv}(g))( \cdot , 0), \qquad \text{where}
$$
$$
g(x,y) = \inf_{z \in E} \left\lbrace \langle G(z), y \rangle + \frac{\lip(G)}{2} \|(x-z,y)\|^2 \right\rbrace + \frac{\lip(G)}{2}\|(x,y)\|^2,
$$
defines such an extension.  We apply this formula to get an extension result for {\em strongly biLipschitz} mappings. Related to the latter, we also consider extensions of $C^{1,1}$ strongly convex functions.
\end{abstract}

\maketitle

\section{An explicit formula for Kirszbraun's theorem}

In 1934 M.D. Kirszbraun \cite{Kirszbraun} proved that, for every subset $E$ of $\R^n$ and every Lipschitz function $f:E\to\R^m$, there exists a Lipschitz extension $F:\R^n\to\R^m$ of $f$ such that $\textrm{Lip}(F)=\textrm{Lip}(f)$. Here $\textrm{Lip}(\varphi)$ denotes the Lipschitz constant of $\varphi$, that is, 
$$\textrm{Lip}(\varphi)=\sup_{x\neq y}\frac{\|\varphi(x)-\varphi(y)\|}{\|x-y\|}.$$
This theorem was generalized for Hilbert spaces $X, Y$ in place of $\R^n$ and $\R^m$ by F.A. Valentine \cite{Valentine} in 1945, and the result is often referred to as the Kirszbraun-Valentine theorem. The proof is rather nonconstructive, in the sense that it requires to use Zorn's lemma or transfinite induction at least in the nonseparable case. In the separable case the proof can be made by induction, considering a dense sequence $\{x_k\}$ in $X$ and  at each step managing to extend $f$ from $E\cup \{x_1, \ldots, x_m\}$ to $E\cup\{x_1, \ldots, x_{m+1}\}$ while preserving the Lipschitz constant of the extension by using Helly's theorem or intersection properties of families of balls, but still it is not clear what the extension looks like. Several other proofs and generalizations which are not constructive either have appeared in the literature; see \cite{Mickle, GrunbaumZarantonello, Federer, BL, ReichSimons, Bauschke2007, AkopyanTarasov}; apart from Zorn's lemma or induction these proofs are based on intersection properties of arbitrary families of balls, or on maximal extensions of non-expansive operators and Fitzpatrick functions. In 2008 H.H. Bauschke and X. Wang \cite{BauschkeWang} gave the first constructive proof of the Kirszbraun-Valentine theorem of which we are aware; they relied on their previous work \cite{BauschkeWangExtensionsOfOperators} on extension and representation of monotone operators and the Fitzpatrick function. See also \cite{AschenbrennerFischer}, where some of these techniques are used to construct definable versions of Helly's and Kirszbraun's theorems in arbitrary definably complete expansions of ordered fields. Finally, in 2015 E. Le Gruyer and T-V. Phan provided sup-inf explicit extension formulas for Lipschitz mappings between finite dimensional spaces by relying on Le Gruyer's solution to the minimal $C^{1,1}$ extension problem for $1$-jets; see \cite[Theorem 32 and 33]{LeGruyer2} and \cite{LeGruyer1}.

In this note we present a short proof of the Kirszbraun-Valentine theorem in which the extension is given by an explicit formula. This proof is based on our previous work concerning $C^{1,1}$ extensions of $1$-jets with optimal Lipschitz constants of the gradients \cite{ALMExplicit}. See \cite{DaniilidisHaddouLeGruyerLey} for an alternative construction of such $C^{1,1}$ extensions on the Hilbert space, and \cite{BrudnyiShvartsman,  Fefferman2005, Fefferman2006, FeffermanSurvey} for the much more difficult question of extending functions (as opposed to jets) to $C^{1,1}$ or $C^{m,1}$ functions on $\R^n$.

\medskip

If $X$ is a Hilbert space, $E \subset X$ is an arbitrary subset and $(f,G): E \to \R \times X$ is a $1$-jet on $E,$ we will say that $(f,G)$ satisfies condition $(W^{1,1})$ with constant $M>0$ on $E$ provided that
\begin{equation}\label{inequalityw11}
f(y) \leq f(x) + \tfrac{1}{2} \langle G(x)+G(y), y-x \rangle + \tfrac{M}{4} \|x-y\|^2 - \tfrac{1}{4M} \|G(x)-G(y)\|^2, \quad \text{for all} \quad x,y\in E. 
\end{equation}

In \cite{Wells, LeGruyer1} it was proved that condition $(W^{1,1})$ with constant $M>0$ is a necessary and sufficient condition on $f: E \to \R, \: G: E \to X$ for the existence of a function $F \in C^{1,1}(X)$ with $\lip(\nabla F) \leq M$ and such that $F=f$ and $\nabla F= G$ on $E.$  Here $\nabla F(x)$ denotes the gradient of $F$ at the point $x,$ that is, the unique vector $\nabla F(x) \in X$ for which $D F(x)(v) = \langle \nabla F(x), v \rangle$ for every $v\in X$, where $DF(x) \in X^*$ denotes the Fr\'{e}chet derivative of $F$ at the point $x.$ More recently, as a consequence of a similar extension theorem for $C^{1,1}$ convex functions, we have found an explicit formula for such an extension $F.$

\begin{theorem}\cite[Theorem 3.4]{ALMExplicit}\label{theoremformulaC11nonnecessaryconvex}
Let $E$ be a subset of a Hilbert space $X$. Given a $1$-jet $(f,G)$ satisfying condition $(W^{1,1})$ with constant $M$ on $E$, the formula
\begin{align*}
& F=\conv(g)- \tfrac{M}{2}\| \cdot \|^2,\\
& g(x) = \inf_{y \in E} \left\lbrace f(y)+\langle G(y), x-y \rangle + \tfrac{M}{2} \|x-y\|^2 \right\rbrace + \tfrac{M}{2}\|x\|^2 , \quad x\in X,
\end{align*}
defines a $C^{1,1}(X)$ function with $F_{|_E}=f$, $(\nabla F)_{|_E} =G$, and $\lip(\nabla F) \leq M$. 
\end{theorem}

Here $\textrm{conv}(g)$ denotes the convex envelope of $g$, defined by
\begin{equation}
\textrm{conv}(g)(x)=\sup\{ h(x) \, : \, h \textrm{ is convex, proper and lower semicontinuous, } h\leq g\}.
\end{equation}
Another expression for $\textrm{conv}(g)$ is given by
\begin{equation}\label{convex envelope as an inf}
\textrm{conv}(g)(x)=\inf\left\lbrace \sum_{j=1}^{k}\lambda_{j} g(x_j) \, : \, \lambda_j\geq 0,
\sum_{j=1}^{k}\lambda_j =1, \, x=\sum_{j=1}^{k}\lambda_j x_j, \, k\in\N \right\rbrace,
\end{equation}
and also by the Fenchel biconjugate of $g$, that is,
\begin{equation}
\textrm{conv}(g)=g^{**},
\end{equation}
where 
\begin{equation}
h^{*}(x):=\sup_{v\in X}\{\langle v, x\rangle -h(v)\};
\end{equation}
see \cite[Proposition 4.4.3]{BorweinVanderwerffbook} for instance. In the case that $X$ is finite dimensional, say $X=\R^n$, the expression \eqref{convex envelope as an inf} can be made simpler: by using Carath\'eodory's Theorem one can show that it is enough to consider convex combinations of at most $n+1$ points. That is to say, if $g:\R^n\to\R$ then
\begin{equation}
\textrm{conv}(g)(x)=\inf\left\lbrace \sum_{j=1}^{n+1}\lambda_{j} g(x_j) \, : \, \lambda_j\geq 0,
\sum_{j=1}^{n+1}\lambda_j =1, \, x=\sum_{j=1}^{n+1}\lambda_j x_j \right\rbrace;
\end{equation}
see \cite[Corollary 17.1.5]{Rockafellar} for instance.

\medskip

In general, the convex envelope does not preserve smoothness of orders higher than $C^1$ and $C^{1,1}.$ For instance, the function $g(x,y)=\sqrt{x^2+e^{-y^2}}$ defined on $\R^2$ is real analytic and its convex envelope is $\conv(g)(x,y)=|x|$ for every $(x,y) \in \R^2;$ see \cite{BenoistUrruty}. In \cite{KirchheimKristensen}, Kirchheim and Kristensen proved that the convex envelope of a differentiable function $g$ on $\R^n$ is of class $C^1,$ provided that $g$ is coercive. On the other hand, if $g$ is of class $C^{1,1}$ on a Hilbert space $X,$ or even if $g$ only satisfies
$$
g(x+h)+g(x-h)-2g(x) \leq M \|h\|^2, \quad x,h \in X
$$
for some $M>0,$ then $\conv(g)$ is of class $C^{1,1}$ and $\lip(\nabla \conv(g)) \leq M;$ see \cite[Theorem 2.3]{ALMExplicit}. In particular, if $g= \inf_i (g_i)$ is the infimum of an arbitrary family of parabolas $g_i,$ whose second derivatives are uniformly bounded by a constant $M>0,$ then $\conv(g)$ is of class $C^{1,1}$ with $\lip(\nabla \conv (g) ) \leq M,$ provided that $g$ has a convex lower semicontinuous minorant. However, $\conv(g)$ is not necessarily of class $C^2$ even when $g$ is the minimum of two parabolas: if we define $g(x)= \min \lbrace x^2, (x-1)^2 \rbrace$ for $x\in \R,$ then $\conv(g)(x)=x^2$ for $x \leq 0, \, \conv(g)=0$ for $0\leq x \leq 1$ and $\conv(g)=(x-1)^2$ for $x\geq 1;$ and therefore $\conv(g) \in C^{1,1}(\R) \setminus C^2(\R).$

\begin{theorem}[Kirszbraun's theorem via an explicit formula]\label{maintheoremkirszbraunformula}
Let $X,Y$ be two Hilbert spaces, $E$ a subset of $X$ and $G: E \to Y$ a Lipschitz mapping. There exists $\widetilde{G} : X \to Y$ with $\widetilde{G}=G$ on $E$ and $\lip(\widetilde{G})=\lip(G).$ In fact, if $M=\lip(G),$ then the function
$$\widetilde{G}(x):= \nabla_Y(\conv(g))(x,0), \quad x\in X, \quad \text{where}
$$
$$
g(x,y) = \inf_{z \in E} \left\lbrace \langle G(z), y \rangle_Y + \tfrac{M}{2} \|x-z\|_X^2 \right\rbrace + \tfrac{M}{2}\|x\|_X^2 + M \|y \|^2_Y, \quad (x,y) \in X \times Y,
$$
defines such an extension.
\end{theorem}

Here $\| \cdot \|_X$ and $\| \cdot \|_Y$ denote the norm on $X$ and $Y$ respectively. Also, the inner products in $X$ and $Y$ are denoted by $\langle \cdot , \cdot \rangle_X$ and $\langle \cdot, \cdot \rangle_Y$ respectively. For any function $F,$ $\nabla_Y F$ will stand for the $Y$-\textit{partial derivatives} of $F,$ that is, the canonical projection from $X \times Y$ onto $Y$ composed with $\nabla F$.
\begin{proof}
We consider on $ X \times Y$ the norm given by $\|(x,y) \| = \sqrt{\| x\|_X^2 + \| y\|_Y^2}$ for every $(x,y) \in X \times Y.$ Then $X \times Y$ is a Hilbert space whose inner product is $\langle (x,y), (x',y') \rangle = \langle x,x' \rangle_X + \langle y,y' \rangle_Y$ for every $(x,y),  (x',y') \in X \times Y.$ We define the $1$-jet $(f^*,G^*)$ on $E \times \lbrace 0 \rbrace\subset X\times Y$ by $f^*(x,0)=0$ and $G^*(x,0)=(0,G(x))$. Then the jet $(f^*,G^*)$ satisfies condition $(W^{1,1})$ on $E \times \lbrace 0 \rbrace$ with constant $M$ (see inequality \eqref{inequalityw11}). Indeed, by the definition of $f^*$ and $G^*$ we can write, for every $(x,0), (y,0) \in E \times \lbrace 0 \rbrace,$ 
\begin{align*}
f^*(x,0)& -f^*(y,0) \\
& + \tfrac{1}{2} \big \langle G^*(x,0)+G^*(y,0), (y,0)-(x,0) \big \rangle + \tfrac{M}{4} \|(x,0)-(y,0)\|^2 - \tfrac{1}{4M} \|G^*(x,0)-G^*(y,0)\|^2 \\
& = \tfrac{1}{2} \big \langle (0,G(x))+(0,G(y)), (y,0)-(x,0) \big \rangle + \tfrac{M}{4} \|x-y\|_X^2 - \tfrac{1}{4M} \|G(x)-G(y)\|_Y^2 \\
& = \tfrac{M}{4} \|x-y\|_X^2 - \tfrac{1}{4M} \|G(x)-G(y)\|_Y^2,
\end{align*}
and the last term is nonnegative because $G$ is $M$-Lipschitz on $E.$

Therefore, Theorem \ref{theoremformulaC11nonnecessaryconvex} asserts that the function $F$ defined by 
\begin{align*}
& F=\textrm{conv}(g)- \tfrac{M}{2}\| \cdot \|^2, \quad \text{where} \quad \\
& g(x,y)=\inf_{z \in E} \left\lbrace f^*(z,0)+\langle G^*(z,0), (x-z,y) \rangle + \tfrac{M}{2} \|(x-z,y)\|^2 \right\rbrace + \tfrac{M}{2}\|(x,y)\|^2,
\end{align*}
is of class $C^{1,1}(X \times Y)$ with $(F, \nabla F)=(f^*,G^*)$ on $E \times \lbrace 0 \rbrace$ and $\lip(\nabla F) \leq M.$ In particular, the mapping $X \ni x \mapsto \widetilde{G}(x): =\nabla_Y F(x,0) \in Y$ is $M$-Lipschitz and extends $G$ from $E$ to $X.$ Finally, the expressions defining $\widetilde{G}$ and $g$ can be simplified as
$$
\widetilde{G}(x) = \nabla_Y \left( \textrm{conv}(g)- \tfrac{M}{2}\| \cdot \|^2 \right) (x,0) = \nabla_Y(\textrm{conv}(g))(x,0) - \nabla_Y \left( \tfrac{M}{2}\| \cdot \|^2 \right)(x,0) = \nabla_Y(\textrm{conv}(g))(x,0)
$$
and
$$
g(x,y) =  \inf_{z \in E} \left\lbrace \langle G(z), y \rangle_Y + \tfrac{M}{2} \|x-z\|_X^2 \right\rbrace + \tfrac{M}{2}\|x\|_X^2 + M \|y \|^2_Y.
$$
\end{proof}

\medskip

Let $X$ be a Hilbert space with inner product and associated norm denoted by $\langle \cdot, \cdot \rangle$ and $\| \cdot \|$ respectively. If $E \subset X$ is arbitrary and $G:E \to X$ is a mapping, we say that $G$ is firmly non-expansive if
$$
\langle G(x)-G(y), x-y \rangle \geq \| G(x)-G(y)\|^2 \quad \text{for all} \quad x,y \in E.
$$

Important examples of firmly non-expansive mappings are the metric projections onto closed convex sets of Hilbert spaces and the \textit{proximal mappings} $\textrm{prox}_f:X \to X$ of proper lower semicontinuous convex functions $f: X \to (-\infty, + \infty];$ see \cite[Chapter 12]{BauschkeCombettesbook}. Firmly non-expansive mappings arise naturally in \textit{convex feasibility problems} too: given a family $C_1,\ldots, C_N$ of closed convex sets of a Hilbert space, find a point $x \in \bigcap_i C_i.$ Also, these mappings are known to be \textit{resolvents} $J_A=(A+ I)^{-1}$ of monotone or maximally monotone operators $A: X \rightrightarrows X,$ and they play a crucial role in the following basic problem that arises in several branches of applied mathematics: given a maximally monotone operator $A: X \rightrightarrows X,$ find a point $x\in X$ with $0 \in Ax.$ For more information about firmly non-expansive mappings and their applications; see \cite{Bauschke2007, BauschkeCombettesbook, BauschkeMoffatWang, BauschkeWang} and the references therein.

\medskip

It is well-known that a mapping $G:E \to X$ is firmly non-expansive if and only $2G-I:E \to X$ is $1$-Lipschitz, where $I$ denotes the identity map; see \cite[Proposition 4.2]{BauschkeCombettesbook} for a proof of this fact. Using this characterization and Theorem \ref{maintheoremkirszbraunformula} we obtain the following corollary.

\begin{corollary}
Let $G: E \to X$ be a firmly non-expansive mapping defined on a subset $E$ of a Hilbert space $X.$ Then $G$ can be extended to a firmly non-expansive mapping $\widetilde{G} : X \to X$ by means of the formula
$$
\widetilde{G}(x):= \tfrac{1}{2} \left( P_2 \left( \nabla(\conv(g))(x,0) \right)+ x \right) \quad x\in X, \quad \text{where}
$$
$$
P_2(x,y)=y, \quad (x,y) \in X \times X, \quad \text{and}
$$
$$
g(x,y) = \inf_{z \in E} \left\lbrace 2 \langle G(z), y \rangle + \tfrac{1}{2} \|z-(x+y)\|^2 \right\rbrace + \tfrac{1}{2} \|x-y\|^2 , \quad (x,y) \in X \times X.
$$
\end{corollary}

\section{Extensions of strongly biLipschitz mappings}

In this section we consider {\em strongly biLipschitz} mappings, which appear naturally as derivatives of strongly convex $C^{1,1}$ functions, and we provide an extension result for this class of mappings.

\begin{defn}\label{defintionstronglybilipschitz}
{\em Let $E$ be a subset of a Hilbert space $X$. We say that a mapping $G:E\to X$ is {\em strongly biLipschitz} provided that
$$
\mbox{SBilip}(G):=\inf_{x,y\in E; \, x\neq y}\frac{2\langle x-y, G(x)-G(y)\rangle}{\|x-y\|^2+\|G(x)-G(y)\|^2}>0.
$$
}
\end{defn}

Strongly biLipschitz mappings may be interesting in regard to some problems in computer vision or image processing where one needs to match points in $\R^n$: for instance, given two sets of points in $\R^n$ with equal cardinality, find a homeomorphism from $\R^n$ onto itself which does not distort distances too much and takes one set onto the other. Supposing that the data satisfy the strongly biLipschitz condition, our explicit formula for such an extension can be useful. Also, in \cite[Corollary 4.5]{BauschkeMoffatWang} it is shown that strongly biLipschitz mappings are closely related to contractive mappings: a maximally monotone mapping $G$ is strongly biLipschitz if and only if its \textit{reflected resolvent} $N= 2(G+I)^{-1}-I $ is a contractive mapping. See \cite[Chapter 12]{RockafellarWets} for more information about resolvent mappings and maximally monotone operators.

It should also be noted that it is not generally true that a biLipschitz mapping whose domain and range is a subset of the same Hilbert space $X$ extends to a total one-to-one continuous mapping, as shown by the following example.
\begin{ex}
{\em Let $| \cdot |$ denote the euclidean norm on $\R^n.$ Let $C=\{x\in\R^n : |x|=1\}\cup\{p\}$, $p$ be any point with $|p|>1$, and $g:C\to\R^n$ be defined by $g(x)=x$ for $|x|=1$ and $g(p)=0$. Then both $g$ and $g^{-1}$ are Lipschitz but no continuous extension of $g$ to $\R^n$ can be one-to-one.}
\end{ex}
However, this is true for the class of strongly biLipschitz mappings, and moreover, the extension can be performed without increasing what seems natural to call the \textit{strong biLipschitz constant} $\mbox{SBilip}(G),$ as we will show by using the extension formula for Lipschitz mappings given by Theorem \ref{maintheoremkirszbraunformula}.

\begin{proposition}
{\em If $G:E\to X$ is strongly biLipschitz then $G$ is biLipschitz.}
\end{proposition}
\begin{proof}
For every $x, y\in E$ we have
\begin{equation}\label{1.2}
2\|x-y\| \|G(x)-G(y)\| \geq 2 \langle x-y, G(x)-G(y)\rangle \geq \alpha \left(\|x-y\|^2+ \|G(x)-G(y)\|^2 \right), 
\end{equation}
where $\alpha=\mbox{SBilip}(G)$ is as in Definition \ref{defintionstronglybilipschitz}. It follows that $G$ is one-to-one. Note that we can write \eqref{1.2} in the equivalent form 
\begin{equation}\label{1.3}
 \big \|G(x)-G(y) - \tfrac{1}{\alpha}(x-y) \big \|^2 \leq \tfrac{1-\alpha^2}{\alpha^2}\|x-y\|^2.
\end{equation}
Setting $\lambda := \dfrac{\|G(x)-G(y)\|}{\|x-y\|}$ for $x \neq y$, \eqref{1.2} holds if and only if
\begin{equation}\label{1.33}
 \lambda^2 - \tfrac{2}{\alpha} \lambda +1 \leq 0,
\end{equation}
which is equivalent to $ \tfrac{1}{K} \leq \lambda \leq K;$ where 
$
K = \tfrac{1}{\alpha} + (\tfrac{1}{\alpha^2}-1)^{1/2}.
$
This means that for any $x, y \in E$ with $x\neq y,$ we have
$$
\dfrac{1}{K} \leq \dfrac{\|G(x)-G(y)\|}{\| x-y\|}  \leq K.
$$
Therefore $G$ is a biLipschitz mapping.
\end{proof}

\begin{rem} 
{\em $(i)$ If $X=\R,$ then the strongly biLipschitz functions are exactly the strictly increasing biLipschitz functions.
\item[] $(ii)$ If $G:E \to X$ is such that $ \mbox{SBilip}(G) =1,$ then $G$ is the restriction of a translation.
\item[] $(iii)$ If $G:E \to X$ is an isometry such that
 $$
\alpha :=\inf_{x,y\in E;\, x \neq y} \dfrac{\langle x-y,G(x)-G(y)\rangle}{\|x-y\|^2} >0,
$$
then $G$ is a strongly biLipschitz function with $\mbox{SBilip}(G) =\alpha$. However the composition of strongly biLipschitz isometries need not be strongly biLipschitz (for instance, if $r:\R^2\to\R^2$ is defined by $r(z)=e^{\pi i/4}z$ then $r$ is strongly biLipschitz, but $r^2$ is not (and in fact $r^4=-\textrm{id}$ is not strongly biLipschitz locally on any disk).}
\end{rem}

\begin{thm}
Let $G:E \to X$ be a strongly biLipschitz mapping. Then $G$ extends to a strongly biLipschitz mapping on $X$ preserving the strongly bilipchitz constant $\SBilip(G).$ Moreover, if $\alpha= \SBilip(G),$ the formula
$$
\widetilde{G}(x):= P_2 \left( \nabla(\conv(g)) (x,0) \right) + \tfrac{1}{\alpha} x, \quad x\in X; \quad \text{where}
$$
$$
P_2(x,y)=y, \quad (x,y) \in X \times X, \quad \text{and}
$$
$$
g(x,y) = \inf_{z \in E} \left\lbrace \langle G(z), y \rangle-\tfrac{1}{\alpha} \langle z,y \rangle + \tfrac{\sqrt{1-\alpha^2}}{2 \alpha} \|x-z\|^2 \right\rbrace + \tfrac{\sqrt{1-\alpha^2}}{\alpha} \left( \tfrac{1}{2}\|x\|^2 +  \|y \|^2 \right) , \quad (x,y) \in X \times X,
$$
defines such an extension.
\end{thm}
\begin{proof}
We know from the characterization \eqref{1.3} that $G-\tfrac{1}{\alpha} I$ is Lipschitz on $E$ with $\textrm{Lip} \left( G-\tfrac{1}{\alpha} I \right) \leq \sqrt{\frac{1-\alpha^2}{\alpha^2}}.$ By Theorem \ref{maintheoremkirszbraunformula}, the mapping $T: X \to X$ defined as 
$$
T(x):= P_2 \left( \nabla(\textrm{conv}(g))  (x,0) \right) \quad x\in X; \quad \text{where}
$$
$$
P_2(x,y)=y \quad \text{for every} \quad (x,y) \in X \times X, \quad \text{and}
$$
$$
g(x,y) = \inf_{z \in E} \left\lbrace \langle G(z), y \rangle-\tfrac{1}{\alpha} \langle z,y \rangle + \tfrac{\sqrt{1-\alpha^2}}{2 \alpha} \|x-z\|^2 \right\rbrace + \tfrac{\sqrt{1-\alpha^2}}{\alpha} \left( \tfrac{1}{2}\|x\|^2 +  \|y \|^2 \right) , \quad (x,y) \in X \times X,
$$
is an extension of $G-\tfrac{1}{\alpha} I$ to all of $X$ such that $\textrm{Lip}(T) =\textrm{Lip} \left( G-\tfrac{1}{\alpha}I \right) \leq \sqrt{\frac{1-\alpha^2}{\alpha^2}}.$ Therefore, if we define the function $\widetilde{G} = T + \tfrac{1}{\alpha} I,$ we have that
$$
\left \| \widetilde{G} (x)-\widetilde{G} (y)- \tfrac{1}{\alpha}(x-y) \right \|^2 = \| T(x)-T(y)\|^2 \leq  \tfrac{1-\alpha^2}{\alpha^2} \| x-y\|^2  \quad \text{for all} \quad x,y\in X.
$$
We obtain from \eqref{1.3} that $\widetilde{G} $ is strongly biLipschitz on $X$ with $\textrm{SBilip}(\widetilde{G})= \alpha.$ Also, since $T$ is an extension of $G-\tfrac{1}{\alpha} I,$ it is obvious that $\widetilde{G}$ is an extension of $G.$ 
\end{proof}

\medskip

\section{$C^{1,1}$ strongly convex functions}

In this section we characterize the $1$-jets which can be interpolated by strongly convex functions of class $C^{1,1}$ in Hilbert spaces. A function $F:X \to \R$ is strongly convex if $F-c\| \cdot\|^2$ is convex for some $c>0.$ In Proposition \ref{remarkpropertiesScw11} below we will see that the gradient of a $C^{1,1}$ strongly convex function is a biLipschitz mapping. These functions arise naturally when studying smooth manifolds of positive curvature as well as in problems involving Monge-Amp\'{e}re equations. See the papers \cite{GhomiJDG2001,GhomiPAMS2002,MinYan} for some results and problems involving smooth strongly convex functions.  

\medskip

Throughout this section $X$ denotes a Hilbert space with norm and inner product denoted by $\langle \cdot, \cdot \rangle$ and $\| \cdot \|$ respectively. 

\begin{definition}\label{definitionstronglyconvexjets}
{\em
Let $E\subseteq X$ be arbitrary, $(f,G):E \to \R \times X$ be a $1$-jet and $c \in \R ,\: M >0$ constants such that $M>c.$ We say that $(f,G)$ satisfies condition $(SCW^{1,1})$ with constants $(c,M)$ provided that
$$
f(x) \geq f(y)+ \langle G (y),x-y \rangle + \tfrac{c}{M-c} \langle G(x)-G(y), y-x \rangle + \tfrac{cM}{2(M-c)} \| x-y\|^2 + \tfrac{1}{2(M-c)}\| G(x)- G(y)\|^2
$$
for every $x,y\in E.$}
\end{definition}

\begin{proposition}\label{remarkpropertiesScw11} {\em
Assume that $(f,G):E \to \R \times X$ satisfies condition $(SCW^{1,1})$ with constants $(c,M).$ Then, the following properties hold.
\item[] $(i)$ For every $x,y \in E$ we have
$$
(c+M)\langle G(x)-G(y), x-y \rangle \geq  cM \| x-y\|^2+\| G(x)- G (y)\|^2.
$$
\item[] $(ii)$ $G$ is Lipschitz with $c \leq \lip(G) \leq M.$

\item[] $(iii)$ If $c>0,$ then $G$ is strongly biLipschitz with $\textrm{SBilip}(G) \geq \frac{2}{c+M} \min \lbrace 1, c M \rbrace.$ 

\item[] $(iv)$ For $c=-M$ we recover Wells's condition $W^{1,1}$ considered in \cite{Wells, LeGruyer1, ALMExplicit}. For $c=0$, $(SCW^{1,1})$ is just condition $(CW^{1,1})$ of \cite{AM, ALMExplicit}. For $c\in (0, M]$ we have what can be called a $C^{1,1}$ strongly convex $1$-jet, which in the extreme case $c=M$ becomes the restriction of a quadratic function to $E$.}
\end{proposition}
\begin{proof} 
$(i)$ Let $x,y\in E.$ By summing the inequalities
$$
f(x) \geq f(y)+ \langle G (y),x-y \rangle + \tfrac{c}{M-c} \langle G(x)-G(y), y-x \rangle + \tfrac{cM}{2(M-c)} \| x-y\|^2 + \tfrac{1}{2(M-c)}\| G(x)- G(y)\|^2
$$
$$
f(y) \geq f(x)+ \langle G (x),y-x \rangle + \tfrac{c}{M-c} \langle G(y)-G(x), x-y \rangle + \tfrac{cM}{2(M-c)} \| x-y\|^2 + \tfrac{1}{2(M-c)}\| G(x)- G(y)\|^2
$$
we obtain
$$
0\geq \left( 1+ \tfrac{2c}{M-c} \right) \langle G(x)-G(y), y-x \rangle + \tfrac{cM}{M-c} \| x-y\|^2 + \tfrac{1}{M-c} \| G(x)- G(y)\|^2,
$$
which is equivalent to the desired estimation.

\item[] $(ii)$ Let $x,y\in E$ be such that $x \neq y.$ Writing $ \lambda = \| G(x)- G(y)\|/ \| x-y\|,$ the inequality in $(i)$ yields $\lambda^2-(c+M)\lambda + cM \leq 0,$ which in turn implies $ c \leq \lambda \leq M.$

\item[] $(iii)$ It follows immediately from $(i)$ and Definition \ref{defintionstronglybilipschitz}.
\end{proof}

We say that a $1$-jet $(f,G):E \to \R \times X$ satisfies condition $(CW^{1,1})$ with constant $M>0$ on $E$ provided that
$$
f(x) \geq f(y)+\langle G(y),x-y \rangle + \tfrac{1}{2M}\| G(x)-G(y)\|^2 \quad \text{for every} \quad x,y\in E.
$$

In \cite[Theorem 2.4]{ALMExplicit} it was shown that $(CW^{1,1})$ is a necessary and sufficient condition on $(f,G)$ for the existence of a $C^{1,1}(X)$ convex extension $F$ of $f$ with $\nabla F=G$ on $E.$ 

\begin{lemma}\label{lemmastronglyconvexconvex}
{\em
The $1$-jet $(f,G)$ satisfies $(SCW^{1,1})$ with constants $(c,M)$ on $E$ if and only if the $1$-jet $(\widetilde{f},\widetilde{G}) = ( f- \tfrac{c}{2}  \| \cdot\|^2, G -c I)$ satisfies condition $(CW^{1,1})$ with constant $M-c$ on $E.$ }
\end{lemma}
\begin{proof}
Assume first that $(f,G)$ satisfies $(SCW^{1,1})$ with constants $(c,M)$ on $E.$ We have 
\begin{align*}
& \widetilde{f}(x)  -\widetilde{f}(y)-\langle \widetilde{G}(y),x-y \rangle -\tfrac{1}{2(M-c)} \| \widetilde{G}(x)-\widetilde{G}(y)\|^2  = f(x) -f(y)-\langle G (y),x-y \rangle - \tfrac{c}{2}\| x-y\|^2  \\
  &   - \tfrac{1}{2(M-c)} \left( \| G(x)-G(y)\|^2 + c^2\| x-y\|^2 + 2c \langle G(x)-G(y), y-x \rangle \right) = f(x) - f(y)- \langle G (y),x-y \rangle \\
&  - \tfrac{c}{M-c} \langle G(x)-G(y), y-x \rangle - \tfrac{cM}{2(M-c)} \| x-y\|^2 - \tfrac{1}{2(M-c)}\| G(x)- G(y)\|^2 \geq 0.
\end{align*}
Conversely, if $(\widetilde{f},\widetilde{G}) = ( f- \tfrac{c}{2}  \| \cdot\|^2, G -c I)$ satisfies condition $(CW^{1,1})$ with constant $M-c,$ we can write
\begin{align*}
 f(x)&  - f(y)-\langle G (y),x-y \rangle - \tfrac{c}{M-c} \langle G(x)-G(y), y-x \rangle - \tfrac{cM}{2(M-c)} \| x-y\|^2 - \tfrac{1}{2(M-c)}\| G(x)- G(y)\|^2 \\
& =\widetilde{f}(x)  -\widetilde{f}(y)-\langle \widetilde{G}(y),x-y \rangle + \tfrac{c}{2} \| x-y\|^2 - \tfrac{c}{M-c} \langle \widetilde{G}(x)-\widetilde{G}(y), y-x \rangle + \tfrac{c^2}{M-c} \|x-y\|^2 \\
& \quad \:  - \tfrac{cM}{2(M-c)} \| x-y\|^2 -  \tfrac{1}{2(M-c)} \left( \| \widetilde{G}(x)-\widetilde{G}(y) \| ^2 + c^2 \|x-y\|^2 + 2c \langle \widetilde{G}(x)-\widetilde{G}(y) , x-y \rangle \right) \\
& = \widetilde{f}(x)  -\widetilde{f}(y)-\langle \widetilde{G}(y),x-y \rangle - \tfrac{1}{2(M-c)}  \| \widetilde{G}(x)-\widetilde{G}(y) \| ^2 \geq 0.
\end{align*}

\end{proof}

\begin{proposition}\label{propositiongeneralpropertiesstronglyconvex}
{\em Let $F\in C^{1,1}(X)$ be such that $\lip(\nabla F) \leq M$ and $g: = F-\frac{c}{2} \| \cdot \|^2$ is a convex function, where $c\in \R, \: M>0$ are two constants (if $c>0,$ we call such a function $F$ a $C^{1,1}$ (globally) strongly convex function). Then the following properties hold.

\item[] $(i)$ We have $M \geq c$ and if $M=c,$ then $F$ is a quadratic function.

\item[] $(ii)$ $g$ is of class $C^{1,1}(X)$ with $\lip(\nabla g) \leq M-c.$

\item[] $(iii)$ Assume that $M>c.$ For every $x,y \in X$ we have
$$
F(x) \geq F(y)+ \langle \nabla F(y),x-y \rangle + \tfrac{c}{M-c} \langle \nabla F(x)-\nabla F(y), y-x \rangle + \tfrac{cM}{2(M-c)} \| x-y\|^2 + \tfrac{1}{2(M-c)}\| \nabla F(x)- \nabla F(y)\|^2.
$$
Therefore, according to Definition \ref{definitionstronglyconvexjets}, $(F, \nabla F)$ satisfies condition $(SCW^{1,1})$ on $X$ with constants $(c,M).$}
\end{proposition}
\begin{proof}

\item[] $(i)$ and $(ii)$ Given $x,h\in X,$ the fact that $\lip(\nabla F) \leq M$ and the convexity of $g$ yield
\begin{align*}
0 & \leq g(x+h)+g(x-h)-2g(x)  \\
& = F(x+h)+F(x-h)-2F(x) - \tfrac{c}{2} \left( \| x+h\|^2+\| x-h\|^2-2\|x\|^2 \right) \leq (M -c)\| h\|^2.
\end{align*}
This shows that $c\leq M$ and, again by convexity of $g,$ that $\lip(\nabla g) \leq M-c.$ Finally observe that if $M=c,$ then $g$ is affine and therefore $F$ is a quadratic function.

\item[] $(iii)$ By virtue of $(ii),$ the convexity of $g$ implies that $(g,\nabla g)$ satisfies condition $(CW^{1,1})$ with constant $M-c$ on $X.$ Thus Lemma \ref{lemmastronglyconvexconvex} gives the desired inequality.

\end{proof}

\begin{theorem}
Let $E$ be an arbitrary subset of a Hilbert space $X$, $f:E \to \R, \: G:E\to X $ be two functions, and $c\in \R, \: M>0$ be two constants. There exists a function $F \in C^{1,1}(X)$ such that $F=f,\: \nabla F= G$ on $E,$ $\lip(\nabla F) \leq M,$ and $F-\tfrac{c}{2} \| \cdot \|^2$ is a convex function if and only if the jet $(f,G)$ satisfies condition $(SCW^{1,1})$ with constants $(c,M)$ on $E.$ In fact, $F$ can be defined by means of the formula
\begin{align*}
& F=\conv(g)+ \tfrac{c}{2}\| \cdot \|^2,\\
& g(x) = \inf_{y \in E} \left\lbrace f(y)+\langle G(y), x-y \rangle + \tfrac{M}{2} \|x-y\|^2 \right\rbrace - \tfrac{c}{2}\|x\|^2 , \quad x\in X.
\end{align*}
Moreover, if $H$ is another function of class $C^{1,1}(X)$ satisfying the above properties, then $H \leq F$ on $X.$
\end{theorem}
\begin{proof}
The necessity of the condition $(SCW^{1,1})$ with constants $(c,M)$ on the jet $(f,G)$ follows immediately from Proposition \ref{propositiongeneralpropertiesstronglyconvex} $(iii).$ Conversely, assume that $(f,G)$ satisfies condition $(SCW^{1,1})$ with constants $(c,M)$ on $E.$ By Lemma \ref{lemmastronglyconvexconvex}, the jet $(\widetilde{f},\widetilde{G}) = ( f- \tfrac{c}{2}  \| \cdot\|^2, G -c I)$ satisfies condition $(CW^{1,1})$ with constant $M-c$ and we can apply \cite[Theorem 2.4]{ALMExplicit} to obtain that
$$
 \widetilde{F}=\conv(g), \quad \text{where}
 $$
 $$
g(x) = \inf_{y \in E} \left\lbrace \widetilde{f}(y)+\langle \widetilde{G}(y), x-y \rangle + \tfrac{M-c}{2} \|x-y\|^2 \right\rbrace = \inf_{y \in E} \left\lbrace f(y)+\langle G(y), x-y \rangle + \tfrac{M}{2} \|x-y\|^2 \right\rbrace - \tfrac{c}{2}\| x\|^2
$$
is convex and of class $C^{1,1}(X)$ with $\lip(\nabla \widetilde{F}) \leq M-c$ and $(\widetilde{F},\nabla \widetilde{F}) = (\widetilde{f},\widetilde{G})$ on $E.$ If we consider the function $F:= \widetilde{F}+ \tfrac{c}{2} \| \cdot \|^2,$ Lemma \ref{lemmastronglyconvexconvex} says that $(F,\nabla F)$ satisfies condition $(SCW^{1,1})$ with constants $(c,M)$ on $X$ (because $(\widetilde{F},\nabla \widetilde{F})$ satisfies $(CW^{1,1})$ with constant $M-c$ on $X$) and $(F, \nabla F )=(f,G)$ on $E.$ It is obvious that $F-\tfrac{c}{2}\| \cdot \|^2$ is convex on $X$ and, by Remark \ref{remarkpropertiesScw11} $(ii),$ $\lip(\nabla F) \leq M.$ 

Finally, if $H$ is a function of class $C^{1,1}(X)$ such that $(H,\nabla H)=(f,G)$ on $E, \: \lip(\nabla H) \leq M$ and $\widetilde{H}:= H- \tfrac{c}{2}\| \cdot \|^2$ is convex, then it is easy to see (using the same calculations as in the proof of Proposition \ref{propositiongeneralpropertiesstronglyconvex} $(ii)$) that $\lip(\nabla \widetilde{H}) \leq M-c$, and obviously $( \widetilde{H},\nabla \widetilde{H})=(\widetilde{f},\widetilde{G})$ on $E.$ We thus have from \cite[Theorem 2.4]{ALMExplicit} that $\widetilde{H} \leq \widetilde{F}$ on $X,$ and therefore $H \leq F$ on $X.$ 
\end{proof}

\end{document}